\documentclass{amsart}

\usepackage{amsmath,amssymb,latexsym}

\numberwithin{equation}{section}

\newtheorem{theorem}{Theorem}
\newtheorem{corollary}{Corollary}
\theoremstyle{definition}
\newtheorem{definition}{Definition}
\newtheorem{problem}{Open Problem}

\newcommand{\Mod}[1]{\ (\mathrm{mod}\ #1)}

\newcommand{\fr}{\frac}

\DeclareMathOperator{\mex}{mex}

\newmuskip\pFqskip
\mathchardef\pFcomma=\mathcode`, 
\newcommand*\pFq[5]{%
  \begingroup
  \begingroup\lccode`~=`,
  \lowercase{\endgroup\def~}{\pFcomma\mkern\pFqskip}%
  \mathcode`,=\string"8000
  {}_{#1}\phi_{#2}\biggl[\genfrac..{0pt}{}{#3}{#4};#5\biggr]%
  \endgroup
}
\newmuskip\pGqskip
\mathchardef\pGcomma=\mathcode`, 

\title{Two-color partitions with evens in one color}

\author{George E. Andrews}
\address{The Pennsylvania State University, University Park, Pennsylvania 16802}
\email{andrews@math.psu.edu}
\thanks{First author partially supported by Simons Foundation Grant 633284.}

\author{Mohamed El Bachraoui}
\address{Dept. Math. Sci, United Arab Emirates University, PO Box 15551, Al-Ain, UAE}
\email{melbachraoui@uaeu.ac.ae}

\date{December 30, 2025}

\begin{document}

\begin{abstract}
We consider sequences counting integer partitions in two colors (red and blue) in which the even parts occur only in blue color. We focus on subsequences defined by constraints on the parity and color of the summands. We establish formulas for our sequences and deduce identities of integer partitions.
\end{abstract}

\maketitle

\section{Introduction}
We adopt the following standard notation from the theory of $q$-series~\cite{Andrews, Gasper-Rahman}
\[
(a;q)_0 = 1,\  (a;q)_n = \prod_{j=0}^{n-1} (1-aq^j),\quad
(a;q)_{\infty} = \prod_{j=0}^{\infty} (1-aq^j),
\]
\[
(a_1,\ldots,a_k;q)_n = \prod_{j=1}^k (a_j;q)_n,\
(a_1,\ldots,a_k;q)_{\infty} = \prod_{j=1}^k (a_j;q)_{\infty},
\]
and
\[
\pFq{2}{1}{a,b}{c}{q, z}:= \sum_{n\geq 0} \fr{(a,b;q)_n z^n}{(q,c;q)_n}.
\]
In this paper we consider sequences of integer partitions in two colors (red and blue)
with conditions on the parity and color of their summands.
We shall write $\lambda_b$ (resp. $\lambda_r$) to denote a part $\lambda$ occurring in blue (resp. red) color
assuming the following order $\lambda_b \geq \lambda_r$.
For references on two-color partitions, see for instance~\cite{Andrews 1987, Andrews 2021, Andrews-Bachraoui 2-color-spt, Andrews-Kumar}.
All of our sequences come from the following set.
\begin{definition}\label{def F-set}
For any nonnegative integer $n$, let $\mathcal{F}(n)$ be the set of two-color integer partitions $n$
such that the even parts may occur only in blue color and let $F(n)=|\mathcal{F}(n)|$.
Then we clearly have
\begin{equation}\label{gen F}
\sum_{n\geq 0} F(n) q^n = \fr{1}{(q;q^2)_\infty^2 (q^2;q^2)_\infty} =1+2q+4q^2+8q^3+14q^4+24 q^5 +\cdots.
\end{equation}
\end{definition}
\begin{definition}\label{def F0-3}
For a nonnegative integer $n$,
let $F_0(n)$ (resp. $F_1(n)$) be the number of partitions
in $\mathcal{F}(n)$ in which the number of odd parts in red color is even (resp. odd).
Furthermore, let $F_2(n)$ (resp. $F_3(n)$) be the number of partitions
in $\mathcal{F}(n)$ in which the number of even parts is even (resp. odd).
\end{definition}
Our first goal in this paper is to establish the following two results regarding $F_0(n)$ and $F_1(n)$.
\begin{theorem}\label{thm mainF0}
There holds
\[
\sum_{n\geq 0} F_0(n) q^n = \fr{(q^{16},-q^6,-q^{10};q^{16})_\infty}{(q;q^2)_\infty (q^2;q^2)_\infty^2}.
\]
\end{theorem}
We get the following immediate consequence of Theorem~\ref{thm mainF0}.
\begin{corollary}\label{main corF0}
For any nonnegative integer $n$, we have that
$F_0(n)$ equals the number of partitions of $n$ in two colors (red and blue), with optional overlines,
wherein the odd parts and the parts $\equiv 0 \Mod{16}$ may appear only in the blue color; the parts $\equiv 6,10 \Mod{16}$ may appear in red and blue, with the first blue occurrence of each such part optionally overlined; and all other even parts may appear in either color.
\end{corollary}
\begin{theorem}\label{thm mainF1}
There holds
\[
\sum_{n\geq 0} F_1(n) q^n = q\fr{(q^{16},-q^2,-q^{14};q^{16})_\infty}{(q;q^2)_\infty (q^2;q^2)_\infty^2}.
\]
\end{theorem}
We get the following immediate corollary of Theorem~\ref{thm mainF1}.
\begin{corollary}\label{main corF1}
For any nonnegative integer $n$, we have that
$F_1(n)$ equals the number of partitions of $n-1$ in two colors (red and blue), with optional overlines,
wherein the odd parts and the parts $\equiv 0 \Mod{16}$ may appear only in the blue color; the parts $\equiv 2,14 \Mod{16}$ may appear in red and blue, with the first blue occurrence of each such part optionally overlined; and all other even parts may appear in either color.
\end{corollary}

For our next two theorems we need the following definitions.
Let $A$ be a positive integer, let $a$ be a nonnegative integer, and let $\pi$ be partition of $n$.
Andrews and Newman recently in~\cite{Andrews-Newman 2020} defined the minimal excludant function $\mex_{A,a}(\pi)$
to be the smallest integer $\equiv a\Mod{A}$ which is not part of $\pi$. Then they defined $p_{A,a}(n)$ (resp. $\overline{p}_{A,a}(n)$)
as the number of partitions $\pi$ of $n$, where
\[ \mex_{A,a}(\pi)\equiv a\Mod{2A}\quad \text{(resp. $\mex_{A,a}(\pi)\equiv A+a\Mod{2A}$)}.
\]
We now extend these concepts to the two-color partitions in the set $\mathcal{F}(n)$ as follows.
Assume that $A$ is even and that the parts $\equiv 0\Mod{\fr{A}{2}}$ may occur only in the blue color.
Let $\mex_{A,a}(\pi, \text{blue})$
be the smallest integer $\equiv a\Mod{A}$ which is not part of $\pi$. Accordingly, let
$p_{A,a}(n, \text{blue})$ (resp. $\overline{p}_{A,a}(n, \text{blue})$)
count the number of partitions $\pi$ in $\mathcal{F}(n)$, where
\[ \mex_{A,a}(\pi, \text{blue})\equiv a\Mod{2A}\quad \text{(resp. $\mex_{A,a}(\pi, \text{blue})\equiv A+a\Mod{2A}$)}.
\]
Our second goal is to prove the following partition identities.
\begin{theorem}\label{thm mainF2}
For any nonnegative integer $n$, we have that $F_2(n)=p_{4,2}(n, \text{blue})$.
\end{theorem}
\begin{theorem}\label{thm mainF3}
For any nonnegative integer $n$, we have that $F_3(n)=\overline{p}_{4,2}(n, \text{blue})$.
\end{theorem}
For example, for $n=4$, we have $F(n)= 14$, enumerating
\[
4_b, 3_b+1_b, 3_b+1_r, 3_r+1_b, 3_r+1_r, 2_b+2_b, 2_b+1_b+1_b, 2_b+1_b+1_r, 2_b+1_r+1_r
\]
\[
1_b+1_b+1_b+1_b, 1_b+1_b+1_b+1_r, 1_b+1_b+1_r+1_r, 1_b+1_r+1_r+1_r, 1_r+1_r+1_r+1_r.
\]
Then it is easy to check that $F_2(4)=10$ and $F_3(4)=4$. Note that if $n=4$, then the part $2$ is
the only possible missing part which is $\equiv 2\Mod{4}$.
Then $p_{4,2}(4, \text{blue})=10$, counting
\[
4_b, 3_b+1_b, 3_b+1_r, 3_r+1_b, 3_r+1_r, 1_b+1_b+1_b+1_b, 1_b+1_b+1_b+1_r, 1_b+1_b+1_r+1_r,
\]
\[
1_b+1_r+1_r+1_r, 1_r+1_r+1_r+1_r.
\]
This confirms that $F_2(4)=p_{4,2}(4, \text{blue})$.
Furthermore, we have $\overline{p}_{4,2}(4, \text{blue})=4$ counting
\[
2_b+2_b, 2_b+1_b+1_b, 2_b+1_b+1_r, 2_b+1_r+1_r.
\]
This shows that $F_3(4)=\overline{p}_{4,2}(4, \text{blue})$.

Finally, it is worth noting that $F(n)$ arises naturally in the theory of integer partitions. Indeed,
by Euler's formula~\cite{Andrews}
\begin{equation}\label{Euler}
(-q;q)_\infty =\fr{1}{(q;q^2)_\infty},
\end{equation}
we get
\begin{align}
\sum_{n\geq 0} F(n) q^n &=\fr{1}{(q;q^2)_\infty^2 (q^2;q^2)_\infty} =\fr{(-q;q)_\infty}{(q;q^2)_\infty (q^2;q^2)_\infty}
= \fr{(-q;q)_\infty}{(q;q)_\infty} \nonumber \\
&=\sum_{n\geq 0}\overline{p}(n) q^n, \label{F-id}
\end{align}
where $\overline{p}(n)$ is the number of overpartitions of $n$.
Here, by an overpartition of $n$ we mean a partition of $n$ where the first occurrence of a part may be overlined,
see~\cite{Corteel-Lovejoy}.
Our third goal in this note is to express some of our sequences and their subsequences in terms of overpartitions.

The rest of the paper is organized as follows.
In Section~\ref{sec overpartitions}, we establish more connections between our sequences and integer overpartitions.
In Section~\ref{sec proof F01} we provide the proof of Theorems~\ref{thm mainF0}-\ref{thm mainF1}.
In Section~\ref{sec proof F23} we give the proof of Theorems~\ref{thm mainF2}-\ref{thm mainF3}.
Finally, Section~\ref{sec comments} is devoted to some comments including
new $q$-series identities and open problems which are related to our work in this paper.
\section{Connections with overpartitions}\label{sec overpartitions}
We will assume that $\overline{p}(k)=0$ if $k$ is not a nonnegative integer.
We start with the following formulas for $F_0(n)$ and $F_1(n)$ in terms of $\overline{p}(n)$.
\begin{theorem}\label{thm F01-id}
For any nonnegative integer $n$ there holds
\begin{align}
F_0(n) &= \fr{\overline{p}(n)+\overline{p}\big(\fr{n}{2}\big)}{2}, \label{F0-id} \\
F_1(n) &= \fr{\overline{p}(n)-\overline{p}\big(\fr{n}{2}\big)}{2}, \label{F1-id}.
\end{align}
\end{theorem}
\begin{proof}
It is easily seen that
\[
\sum_{n\geq 0} \big( F_0(n)-F_1(n) \big) q^n =\fr{1}{(-q;q^2)_\infty (q;q^2)_\infty (q^2;q^2)_\infty}.
\]
Then by~\eqref{Euler} and~\eqref{F-id}, we obtain
\begin{align}
\sum_{n\geq 0} \big( F_0(n)-F_1(n) \big) q^n &=\fr{1}{(-q;q^2)_\infty (q;q^2)_\infty (q^2;q^2)_\infty} \nonumber \\
&=\fr{(-q;q)_\infty}{(-q;q^2)_\infty (q^2;q^2)_\infty}  \nonumber\\
&=\fr{(-q^2;q^2)_\infty}{(q^2;q^2)_\infty} \nonumber \\
&=\sum_{n\geq 0} \overline{p}(n) q^{2n} . \label{G0-G1}
\end{align}
Furthermore, we clearly have
\[
\sum_{n\geq 0} \big( F_0(n)+F_1(n) \big) q^n = \sum_{n\geq 0} F(n) q^n,
\]
and thus by~\eqref{F-id},
\begin{equation}\label{gen G0pG1}
\sum_{n\geq 0} \big( F_0(n)+F_1(n) \big) q^n = \sum_{n\geq 0}\overline{p}(n) q^n.
\end{equation}
Now combine~\eqref{G0-G1} with~\eqref{gen G0pG1} to achieve the desired formulas.
\end{proof}
We now introduce a subset of $\mathcal{F}(n)$.
\begin{definition}\label{def H-set}
For any nonnegative integer $n$,
let $\mathcal{H}(n)$ be the subset of $\mathcal{F}(n)$ wherein the parts of the same color do not repeat and let $H(n)=|\mathcal{H}(n)|$.
Then it is easy to see that
\begin{equation}\label{gen H}
\sum_{n\geq 0} H(n) q^n = (-q;q^2)_\infty^2 (-q^2;q^2)_\infty = 1+2q+2q^2+4q^3+6q^4+8q^5+\cdots.
\end{equation}
\end{definition}
Now by~\eqref{Euler}, we find
\begin{align}
\sum_{n\geq 0} H(n) q^n &= (-q;q^2)_\infty^2 (-q^2;q^2)_\infty \nonumber \\
&=\fr{(-q;q^2)_\infty}{(q;q^2)_\infty} \nonumber \\
& = \sum_{n\geq 0}\overline{p}_o(n) q^n, \label{H-id}
\end{align}
where $\overline{p}_o(n)$ is the number of overpartitions of $n$ into odd parts,~\cite{Chen, Hirschhorn-Sellers}.
For example, for $n=4$ we have
$H(4)=6$ counting
\[
4_b, 3_b+1_b, 3_b+1_r, 3_r+1_b,3_r+1_r,2_b+1_b+1_r
\]
and $\overline{p}_o(4)=6$, counting
\[
\bar{3}+\bar{1}, \bar{3}+1, 3+\bar{1}, 3+1, \bar{1}+1+1+1, 1+1+1+1.
\]
\begin{definition}
For a nonnegative integer $n$, let $H_0(n)$ (resp. $H_1(n)$) be the number of partitions
counted by $H(n)$ in which the number of even parts is even (resp. odd). Furthermore, let
$H_2(n)$ (resp. $H_3(n)$) be the number of partitions
counted by $H(n)$ in which the number of parts is even (resp. odd).
\end{definition}
\begin{theorem}\label{thm H-wt}
For any nonnegative integer $n$, there holds
\[
\begin{split}
(a)\ & H_0(n) =\begin{cases} \fr{\overline{p}_o(n)}{2} +1 & \text{if $n$ is a square,} \\
\fr{\overline{p}_o(n)}{2} & \text{otherwise,}
\end{cases}\\
(b)\ & H_1(n) =\begin{cases} \fr{\overline{p}_o(n)}{2}-1 & \text{if $n$ is a square,} \\
\fr{\overline{p}_o(n)}{2} & \text{otherwise,}
\end{cases} \\
(c)\ & H_2(n) =\begin{cases} \fr{\overline{p}_o(n)}{2} +(-1)^n & \text{if $n$ is a square,} \\
\fr{\overline{p}_o(n)}{2} & \text{otherwise,}
\end{cases} \\
(d)\ & H_3(n) =\begin{cases} \fr{\overline{p}_o(n)}{2}-(-1)^n & \text{if $n$ is a square,} \\
\fr{\overline{p}_o(n)}{2} & \text{otherwise.}
\end{cases}
\end{split}
\]
\end{theorem}
\begin{proof}
(a, b) It is clear that
\begin{equation*}
\sum_{n\geq 0} \big(H_0(n)-H_1(n)\big) q^n = (q^2;q^2)_\infty (-q;q^2)_\infty^2,
\end{equation*}
which by~\eqref{Jacobi3} applied with $q\to q^2$ and $x=-q^{-1}$ yields
\begin{equation}\label{H0-H1 0}
\sum_{n\geq 0} \big(H_0(n)-H_1(n)\big) q^n = 1+2\sum_{n=1}^{\infty}q^{n^2}.
\end{equation}
Furthermore, by~\eqref{gen H} and~\eqref{H-id}
\begin{align}
\sum_{n\geq 0} \big(H_0(n)+H_1(n)\big) q^n
&= (-q;q^2)_\infty^2 (-q^2;q^2)_\infty \nonumber \\
&= \sum_{n\geq 0} H(n) q^n \nonumber \\
&=\sum_{n=0}^{\infty} \overline{p}_o(n) q^n. \label{H0+H1 0}
\end{align}
Then the desired identities for $H_0(n)$ and $H_1(n)$ follow by combining~\eqref{H0-H1 0} and~\eqref{H0+H1 0}.

(c, d) Similarly, with the help of~\eqref{Jacobi3} applied to $q\to q^2$ and $x=-q^{-1}$, we get
\begin{equation}\label{H2-H3}
\sum_{n\geq 0} \big(H_2(n)-H_3(n)\big) q^n = (q^2;q^2)_\infty (q;q^2)_\infty^2
=1+2\sum_{n=1}^{\infty}(-1)^n q^{n^2}.
\end{equation}
In addition, with the help of~\eqref{H-id},
\begin{equation}\label{H2pH3}
\sum_{n\geq 0} \big(H_2(n)+H_3(n)\big) q^n
= \sum_{n\geq 0} H(n) q^n = \sum_{n=0}^\infty \overline{p}_o(n) q^n.
\end{equation}
Now use~\eqref{H2-H3} and~\eqref{H2pH3}
to derive the desired identities for $H_2(n)$ and $H_3(n)$.
\end{proof}
\section{Proof of Theorems~\ref{thm mainF0}-\ref{thm mainF1}}\label{sec proof F01}
\emph{Proof of Theorem~\ref{thm mainF0}.\ }
It is easily seen that
\begin{align}
\sum_{n\geq 0} \big(F_0(n)-F_1(n)\big) q^n &= \fr{1}{(-q;q^2)_\infty (q;q^2)_\infty (q^2;q^2)_\infty}, \nonumber \\
&=\fr{1}{(-q;q^2)_\infty (q;q)_\infty}, \label{F0-F1} \\
\sum_{n\geq 0} \big(F_0(n)+F_1(n)\big) q^n &= \fr{1}{(q;q^2)_\infty (q;q^2)_\infty (q^2;q^2)_\infty}, \nonumber \\
&= \fr{1}{(q;q^2)_\infty (q;q)_\infty}. \label{F0pF1}
\end{align}
Then upon adding~\eqref{F0-F1} and~\eqref{F0pF1}, we find
\begin{align}
\sum_{n\geq 0} F_0(n) q^n &= \fr{1}{2}\fr{1}{(q;q)_\infty} \Big(\fr{1}{(q;q^2)_\infty} + \fr{1}{(-q;q^2)_\infty} \Big) \nonumber \\
&=\fr{1}{2}\fr{1}{(q;q)_\infty (q^2;q^4)_\infty} \big( (-q;q^2)_\infty + (q;q^2)_\infty \big) \nonumber \\
&=\fr{1}{2}\fr{1}{(q;q)_\infty (q^2;q^2)_\infty} \big( (-q,-q^3,q^4;q^4)_\infty + (q,q^3,q^4;q^4)_\infty \big) \nonumber \\
&=\fr{1}{2}\fr{1}{(q;q)_\infty (q^2;q^2)_\infty} \sum_{n=-\infty}^\infty q^{2n^2+n} \big( 1+(-1)^n \big) \label{F0-hlp-1}\\
&= \fr{1}{(q;q)_\infty (q^2;q^2)_\infty} \sum_{n=-\infty}^\infty q^{8n^2 + 2n} \nonumber \\
&= \fr{(q^{16},-q^6,-q^{10};q^{16})_\infty}{(q;q)_\infty (q^2;q^2)_\infty} \label{F0-hlp-2} \\
&= \fr{(q^{16},-q^6,-q^{10};q^{16})_\infty}{(q;q^2)_\infty (q^2;q^2)_\infty^2}, \nonumber
\end{align}
where in~\eqref{F0-hlp-1} we applied~\eqref{Jacobi3} with $q\to q^4$ and $x=q^{-1}$ and
in~\eqref{F0-hlp-2} we applied~\eqref{Jacobi3} with $q\to q^{16}$ and $x= q^{-6}$. This completes the proof.

\emph{Proof of Theorem~\ref{thm mainF1}.\ }
By subtracting~\eqref{F0-F1} from~\eqref{F0pF1}, we similarly get
\begin{align*}
\sum_{n\geq 0} F_1(n) q^n &= \fr{1}{2}\fr{1}{(q;q)_\infty} \Big(\fr{1}{(q;q^2)_\infty} - \fr{1}{(-q;q^2)_\infty} \Big) \\
&=\fr{1}{2}\fr{1}{(q;q)_\infty (q^2;q^2)_\infty} \sum_{n=-\infty}^\infty q^{2n^2+n} \big( 1-(-1)^n \big) \\
&= \fr{1}{(q;q)_\infty (q^2;q^2)_\infty} \sum_{n=-\infty}^\infty q^{8n^2 + 6n+1}  \\
&= q\fr{(q^{16},-q^2,-q^{14};q^{16})_\infty}{(q;q^2)_\infty (q^2;q^2)_\infty^2}.
\end{align*}
In the last step we applied~\eqref{Jacobi3} with $q\to q^{16}$ and $x=q^{-2}$.
This completes the proof.
\section{Proof of Theorems~\ref{thm mainF2}-\ref{thm mainF3}}\label{sec proof F23}
We need Jacobi's triple product identity~\cite[p. 12]{Gasper-Rahman}
\begin{equation}\label{Jacobi3}
(q,-xq,-1/x;q)_\infty = \sum_{n=-\infty}^{\infty} x^n q^{\fr{n(n+1)}{2}}.
\end{equation}
\emph{Proof of Theorem~\ref{thm mainF2}.\ }
It is clear that
\begin{align}
\sum_{n\geq 0} \big(F_2(n)-F_3(n)\big) q^n &= \fr{1}{(-q^2;q^2)_\infty (q;q^2)_\infty^2}, \label{F2-F3} \\
\sum_{n\geq 0} \big(F_2(n)+F_3(n)\big) q^n &= \fr{1}{(q^2;q^2)_\infty (q;q^2)_\infty^2}. \label{F2pF3}
\end{align}
Then by adding~\eqref{F2-F3} and~\eqref{F2pF3} we get
\begin{align}
\sum_{n\geq 0} F_2(n) q^n
&=\fr{1}{2} \Big( \fr{1}{(q^2;q^2)_\infty (q;q^2)_\infty^2} + \fr{1}{(-q^2;q^2)_\infty (q;q^2)_\infty^2} \Big) \nonumber \\
&=\fr{1}{2 (q^2;q^2)_\infty (q;q^2)_\infty^2} \Big( 1+ \fr{(q^2;q^2)_\infty}{(-q^2;q^2)_\infty} \Big)\nonumber \\
&=\fr{1}{2 (q^2;q^2)_\infty (q;q^2)_\infty^2} \big( 1+ (q^2;q^2)_\infty (q^2;q^4)_\infty \big) \label{Euler-hlp1} \\
&=\fr{1}{2 (q^2;q^2)_\infty (q;q^2)_\infty^2} \big( 1+ (q^4,q^2,q^2;q^4)_\infty \big) \nonumber\\
&=\fr{1}{2 (q^2;q^2)_\infty (q;q^2)_\infty^2} \Big( 1 +\sum_{n=-\infty}^\infty (-1)^n q^{2n^2} \Big) \label{Jacobi-hlp1} \\
&=\fr{1}{(q^2;q^2)_\infty (q;q^2)_\infty^2} \sum_{n\geq 0}(-1)^n q^{2n^2} \nonumber \\
&=\fr{1}{(q^2;q^2)_\infty (q;q^2)_\infty^2} \sum_{n\geq 0} q^{8n^2} (1-q^{8n+2}) \nonumber \\
&=\fr{1}{(q^4;q^4)_\infty (q;q^2)_\infty^2} \sum_{n\geq 0} \fr{q^{8n^2} (1-q^{8n+2})}{(q^2;q^4)_\infty} \nonumber \\
&=\fr{1}{(q^4;q^4)_\infty (q;q^2)_\infty^2} \sum_{n\geq 0} \fr{q^{2+6+10+\ldots+(8n-2)}}
{\displaystyle{\prod_{\substack{j=0 \\ j\ne 2n}}^{\infty}} (1-q^{4j+2})}, \nonumber
\end{align}
where in step~\eqref{Euler-hlp1} we used~\eqref{Euler} and in step~\eqref{Jacobi-hlp1} we applied~\eqref{Jacobi3} with
$q\to q^4$ and $x=-q^{-2}$.
This completes the proof.

\emph{Proof of Theorem~\ref{thm mainF3}.\ } Similarly, by subtracting~\eqref{F2-F3} from~\eqref{F2pF3}, we obtain
\begin{align*}
\sum_{n\geq 0} F_3(n) q^n
&=\fr{1}{2} \Big( \fr{1}{(q^2;q^2)_\infty (q;q^2)_\infty^2} - \fr{1}{(-q^2;q^2)_\infty (q;q^2)_\infty^2} \Big) \\
&= \fr{1}{(q^2;q^2)_\infty (q;q^2)_\infty^2}  \sum_{n\geq 1}(-1)^n q^{2n^2}  \\
&= \fr{1}{(q^2;q^2)_\infty (q;q^2)_\infty^2}  \sum_{n\geq 0}(-1)^{n+1} q^{2(n+1)^2}  \\
&=\fr{1}{(q^2;q^2)_\infty (q;q^2)_\infty^2} \sum_{n\geq 0} q^{8n^2+8n+2} (1-q^{8n+6}) \\
&=\fr{1}{(q^4;q^4)_\infty (q;q^2)_\infty^2} \sum_{n\geq 0} \fr{q^{2+6+10+\ldots+(8n+2)}}
{\displaystyle{\prod_{\substack{j=0 \\ j\ne 2n+1}}^{\infty}} (1-q^{4j+2})}.
\end{align*}
This completes the proof.
\section{Concluding remarks}\label{sec comments}
{\bf 1.\ } Our proofs for Corollaries~\ref{main corF0}-\ref{main corF1} and for Theorems~\ref{thm mainF2}-\ref{thm mainF3} are
analytic, relying on $q$-series manipulations.
\begin{problem}
Give bijective proofs for Corollaries~\ref{main corF0}-\ref{main corF1} and for Theorems~\ref{thm mainF2}-\ref{thm mainF3}.
\end{problem}

{\bf 2.\ } Recall that $F(n)$ is the number of two-color integer partitions $n$
such that the even parts may occur only in blue color.
Now, by separating the cases according to whether the greatest part is even--blue, odd--blue, or odd--red,
and keeping in mind the assumed order $\lambda_b>\lambda_r$, we find that
\begin{align}
\sum_{n\geq 0} F(n)q^n &= \sum_{n\geq 0}\fr{q^{2n}}{(q^2;q^2)_n (q;q^2)_{n}^2}\quad (\text{greatest part even}) \nonumber \\
&+ \sum_{n\geq 0} \fr{q^{2n+1}}{(q^2;q^2)_n (q;q^2)_{n+1}^2} \quad (\text{greatest part odd-blue}) \nonumber \\
&+ \sum_{n\geq 0} \fr{q^{2n+1}}{(q^2;q^2)_n (q;q^2)_{n+1}(q;q^2)_{n}}\quad (\text{greatest part odd-red}).\label{F-gpt}
\end{align}
On the other hand, if we consider the smallest part instead of the greatest part, we get
\begin{align}
\sum_{n\geq 0} F(n)q^n &= \sum_{n\geq 0}\fr{q^{2n+2}}{(q^{2n+2},q^{2n+3},q^{2n+3};q^2)_{\infty}}\quad (\text{smallest part even}) \nonumber\\
&+ \sum_{n\geq 0} \fr{q^{2n+1}}{(q^{2n+2},q^{2n+1},q^{2n+3};q^2)_{\infty}} \quad (\text{smallest part odd-blue}) \nonumber\\
&+ \sum_{n\geq 0} \fr{q^{2n+1}}{(q^{2n+2},q^{2n+1},q^{2n+1};q^2)_{\infty}}
\quad (\text{smallest part odd-red}).\label{F-spt}
\end{align}
Thus, by a combination of~\eqref{F-gpt}-\eqref{F-spt} and~\eqref{gen F}, we get the following identities of $q$-series.
\begin{align}
&\sum_{n\geq 0}\fr{q^{2n}}{(q^2;q^2)_n (q;q^2)_{n}^2} \Big(1+\fr{q}{1-q^{2n+1}}+\fr{q}{(1-q^{2n+1})^2}\Big) \nonumber \\
&=1+\sum_{n\geq 0}\fr{q^{2n+1}}{(q^{2n+2};q^2)_\infty (q^{2n+3};q^2)_\infty^2} \Big(q+\fr{1}{1-q^{2n+1}}+\fr{1}{(1-q^{2n+1})^2}\Big) \nonumber \\
&=\fr{1}{(q^2;q^2)_\infty (q;q^2)_\infty^2} \label{gen F2}.
\end{align}

Similarly,  recall that $H(n)$ is the number of two-color integer partitions $n$
such that the even parts may occur only in blue color and the parts of the same color are not repeated.
Now, separating the cases according to whether the greatest part is even--blue, odd--blue, or odd--red, we find that
\begin{align}
\sum_{n\geq 1} H(n)q^n &= 1+ \sum_{n\geq 0}q^{2n+2}(-q^2;q^2)_{n} (-q;q^2)_{n+1}^2 \quad (\text{greatest part even}) \nonumber \\
&+ \sum_{n\geq 0} q^{2n+1}(-q^2;q^2)_n (-q;q^2)_{n}^2 \quad (\text{greatest part odd-red}) \nonumber \\
&+ \sum_{n\geq 0} q^{2n+1}(-q^2;q^2)_n (-q;q^2)_{n+1}(-q;q^2)_{n} \quad (\text{greatest part odd-blue}).\label{H-gpt}
\end{align}
Taking into account the smallest part instead of the greatest part, we deduce
\begin{align}
\sum_{n\geq 1} H(n)q^n &= 1+ \sum_{n\geq 0}q^{2n+2}(-q^{2n+4},-q^{2n+3},-q^{2n+3};q^2)_{\infty} \quad (\text{smallest part even}) \nonumber \\
&+ \sum_{n\geq 0} q^{2n+1}(-q^{2n+2},-q^{2n+1},-q^{2n+3};q^2)_{\infty} \quad (\text{smallest part odd-red}) \nonumber \\
&+ \sum_{n\geq 0} q^{2n+1}(-q^{2n+2},-q^{2n+3},-q^{2n+3};q^2)_{\infty} \quad (\text{smallest part odd-blue}).\label{H-spt}
\end{align}
Thus, combining~\eqref{H-gpt}-\eqref{H-spt} with~\eqref{gen H}, we derive the following identities of $q$-series
\begin{align}
& \sum_{n\geq 0}q^{2n+1}(-q^2;q^2)_{n} (-q;q^2)_{n+1}^2
\Big( q+\fr{1}{1+q^{2n+1}}+\fr{1}{(1+q^{2n+1})^2} \Big) \nonumber \\
&=\sum_{n\geq 0} q^{2n+1} (-q^{2n+1},-q^{2n+2},-q^{2n+3};q^2)_\infty \Big(1+\fr{1}{1+q^{2n+1}}+\fr{q}{(1+q^{2n+1})(1+q^{2n+1})} \Big) \nonumber \\
&=(-q^2;q^2)_\infty (-q;q^2)_{\infty}^2 -1. \label{gen H2}
\end{align}
\begin{problem}
Give analytic proofs for the formulas listed in~\eqref{gen F2} and~\eqref{gen H2}.
\end{problem}
%

\medskip
\noindent{\bf Data Availability Statement.} Not applicable.

\medskip
\noindent{\bf Declarations.} The authors state that there is no conflict of interest.

\end{document}